\documentclass{article}
\usepackage[utf8]{inputenc}
\usepackage[margin=1 in]{geometry}
\usepackage{amsmath}
\usepackage{amsfonts}
\usepackage{amssymb}
\usepackage{amsthm}
\usepackage{mathrsfs}
\usepackage{tikz-cd}
\usepackage{enumitem}
\usepackage{lmodern}
\usepackage{hyperref}
\usepackage[vcentermath,enableskew]{youngtab}
\usepackage{ytableau}

\usepackage{graphbox}

\newtheorem{theorem}{Theorem}
\numberwithin{theorem}{section}
\newtheorem{lemma}[theorem]{Lemma}
\newtheorem{prop}[theorem]{Proposition}
\newtheorem{corollary}[theorem]{Corollary}

\makeatletter
\let\c@theorem\c@figure
\makeatother

\newtheorem*{prop*}{Proposition}

\theoremstyle{definition}

\newtheorem{example}[theorem]{Example}

\newcommand{\tb}{\textbf}

\newcommand{\Z}{\mathbb{Z}}

\newcommand{\tn}{\textnormal}
\newcommand{\se}{\subseteq}

\newcommand{\til}{\widetilde}

\newcommand{\bs}{\backslash}

\newcommand{\ol}{\overline}
\newcommand{\lam}{\lambda}
\newcommand{\m}{\ol{\til{m}}}

\newcommand{\ims}{\includegraphics[width=0.8cm,align=c]} 
\newcommand{\imxs}{\includegraphics[width=0.6cm,align=c]} 

\newcommand{\tcb}{\textcolor{blue}}

\usepackage[maxbibnames=10,style=alphabetic]{biblatex}
\title{A power sum expansion for the Kromatic symmetric function}
\author{Laura Pierson \\ University of Waterloo \\ \href{mailto:lcpierson73@gmail.com}{lcpierson73@gmail.com}}

\begin{document}

\maketitle

\begin{abstract}
    The \emph{\tb{\tcb{chromatic symmetric}}} $X_G$ function is a symmetric function generalization of the chromatic polynomial of a graph, introduced by Stanley \cite{stanley1995symmetric}. Stanley \cite{stanley1995symmetric} gave an expansion formula for $X_G$ in terms of the \emph{\tb{\tcb{power sum symmetric functions}}} $p_\lam$ using the principle of inclusion-exclusion, and Bernardi and Nadeau \cite{bernardi2020combinatorial} gave an alternate $p$-expansion for $X_G$ in terms of acyclic orientations. Crew, Pechenik, and Spirkl \cite{kromatic2023} defined the \emph{\tb{\tcb{Kromatic symmetric function}}} $\ol{X}_G$ as a $K$-theoretic analogue of $X_G$, constructed in the same way except that each vertex is assigned a nonempty set of colors such that adjacent vertices have nonoverlapping color sets. They defined a $K$-analogue $\ol{p}_\lam$ of the power sum basis and computed the first few coefficients of the $\ol{p}$-expansion of $\ol{X}_G$ for some small graphs $G$. They conjectured that the $\ol{p}$-expansion always has integer coefficients and asked whether there is an explicit formula for these coefficients. In this note, we give a formula for the $\ol{p}$-expansion of $\ol{X}_G$, show two ways to compute the coefficients recursively (along with examples), and prove that the coefficients are indeed always integers. In a more recent paper \cite{new_p_expansion}, we use our formula from this note to give a combinatorial description of the $\ol{p}$-coefficients $[\ol{p}_\lam]\ol{X}_G$ and a simple characterization of their signs in the case of unweighted graphs.
\end{abstract}

\section{Introduction}\label{sec:intro}

Let $G$ be a graph with vertex set $V$. The \emph{\tb{\tcb{chromatic symmetric function}}} $X_G$ was introduced by Stanley in \cite{stanley1995symmetric} as a symmetric function generalization of the chromatic polynomial \cite{birkhoff1913reducibility}. Stanley gave expansion formulas for $X_G$ in a few different bases, including the \emph{\tb{\tcb{power sum symmetric functions}}} $p_\lam.$ His classic $p$-expansion formula is an alternating sum over all subsets of the edges of $G$, but he also gave an alternative formulation in terms of the broken circuit complex of $G$, which unlike his subset sum expansion does not involve extra cancellation. Bernardi and Nadeau \cite{bernardi2020combinatorial} gave an alternative interpretation of the $p$-expansion of $X_G$ in terms of the source components of its acyclic orientations.

The \emph{\tb{\tcb{Kromatic symmetric function}}} $\ol{X}_G$ was defined by Crew, Pechenik, and Spirkl in \cite{kromatic2023} as a $K$-theoretic analogue of $X_G.$ The idea of $K$-theory is to deform the cohomology ring of a topological space by introducing an extra parameter $\beta$ (often set to $-1$), so part of the motivation for defining $\ol{X}_G$ was to help shed light on possible topological interpretations for $X_G.$ Marberg \cite{marberg2023kromatic} gave an alternative construction of $\ol{X}_G$ in terms of linearly compact Hopf algebras and showed that it arises as a natural analogue of $X_G$ in that setting as well.

One natural question about $\ol{X}_G$, posed by the authors of \cite{kromatic2023}, is whether there is a $K$-analogue of Stanley's power sum expansion formula in terms of an appropriate $K$-analogue of the $p$-basis. The authors of \cite{kromatic2023} defined such a $K$-analogue of the $p$-basis, which they called the $\ol{p}$-basis. They gave a table listing the first few coefficients of the $\ol{p}$-expansions of $\ol{X}_G$ for several small graphs $G$, which they computed using Sage \cite{sage}. The leading terms (i.e. the terms where $|\lam|$ is minimal), must always match the corresponding terms from ordinary $p$-expansion of the ordinary chromatic symmetric function $X_G.$ (By leading terms, we mean the terms where $|\lam| = |V|$ for an unweighted graph, or where $|\lam| = \omega(G) := \sum_{v\in V} \omega(v)$ for a weighted graph with weight function $\omega:V\to \Z_{>0}$, corresponding to each vertex getting one color.)

Unlike the $p$-expansion of $X_G,$ they found that the $\ol{p}$-expansions for $\ol{X}_G$ appeared to have infinitely many terms, but the coefficients still seemed to always be integers. For instance, below are first few terms (grouped by $|\lam|$) of the $\ol{p}$-expansions for the single edge $\ims{figs/e}$, and for the graph $K_{21}$ consisting of a single edge connecting a vertex of weight 2 and a vertex of weight 1:
\begin{align*}
    \ol{X}_{K_{21}} &= (- \ol{p}_3 + \ol{p}_{21} ) + (\ol{p}_4 - \ol{p}_{31}) + (\ol{p}_{41} - \ol{p}_{32}) + (- \ol{p}_6 + \ol{p}_{42} + \ol{p}_{33} - \ol{p}_{321}) \\
    & \ \ \ + (-\ol{p}_{61} + \ol{p}_{421} + \ol{p}_{331}) + (\ol{p}_8 - \ol{p}_{62} + \ol{p}_{332}) + (\ol{p}_{81} + \ol{p}_{63} - \ol{p}_{621} - \ol{p}_{333} + \ol{p}_{3321}) + \dots\\ \\
    \ol{X}_{\imxs{figs/e}} &= (- \ol{p}_2 + \ol{p}_{11}) + (2\ol{p}_3 - 2\ol{p}_{21}) + (-4\ol{p}_4 + 4 \ol{p}_{31} + \ol{p}_{22} - \ol{p}_{211}) \\
    & \ \ \ + (6\ol{p}_5 - 8 \ol{p}_{41} - 3\ol{p}_{32} + 2 \ol{p}_{311} + 2\ol{p}_{221}) \\
    & \ \ \ + (-9\ol{p}_6 + 12\ol{p}_{51} + 4\ol{p}_{42} - 4\ol{p}_{411} + \ol{p}_{33} - 4\ol{p}_{321} - \ol{p}_{222} + \ol{p}_{2211})\dots.
\end{align*}
Since $\ims{figs/e}$ has $|V| = 2,$ the leading terms are $-\ol{p}_2 + \ol{p}_{11},$ matching the ordinary $p$-expansion $X_{\imxs{figs/e}} = - p_2 + p_{11}.$ Similarly, $K_{21}$ has weight $\omega(K_{12}) = 2+1=3,$ so the leading terms are $- \ol{p}_3 + \ol{p}_{21},$ which matches the ordinary $p$-expansion $X_{K_{21}} = - p_3 + p_{21}.$ We will explain where the rest of the coefficients in these expansions come from when we revisit these examples in \S \ref{sec:examples}.

The authors of \cite{kromatic2023} asked whether there is an explicit formula for the $\ol{p}$-expansion of $\ol{X}_G$. Here, we give such a formula that applies to any weighted graph $G$ with a vertex weight function $\omega:V \to \Z_{>0},$ in terms of its \emph{\tb{\tcb{independence polynomial}}} $$I_{(G,\omega)}(t) := \sum_{S\se V\text{ an independent set}}t^{\omega(S)},$$ where the weight of a subset $S$ is $\omega(S):= \sum_{v\in S}\omega(v).$ We write $G|_W$ to denote the induced subgraph of $G$ with vertex set $W.$

\begin{theorem}\label{thm:p_expansion}
    $\ol{X}_{(G,\omega)}$ satisfies the formula $$\ol{X}_{(G,\omega)} = \sum_{W\se V} (-1)^{|V\bs W|}\prod_{k\ge 1}(1+\ol{p}_k)^{a_W(k)},$$ where the exponents $a_W(k)$ are the unique (sometimes negative) real numbers satisfying $$(1+t)^{a_W(1)}(1+t^2)^{a_W(2)}(1+t^3)^{a_W(3)}\dots = I_{(G|_W,\omega)}(t).$$
\end{theorem}

The idea for where the $(1+t^k)$ terms come from is the factorization $1 + \ol{p}_k = \prod_{i \ge 1}(1 + x_i^k),$ which essentially lets us reduce to a single variable problem.

As a consequence of Theorem \ref{thm:p_expansion}, we can write down an explicit formula for the coefficients $[\ol{p}_\lam]\ol{X}_{(G,\omega)}$:

\begin{corollary}\label{cor:p_coeff}
    The coefficient of $\ol{p}_\lam$ in $\ol{X}_{(G,\omega)}$ is given by $$[\ol{p}_\lam]\ol{X}_{(G,\omega)} = \sum_{W\se V}(-1)^{|V\bs W|} \prod_{k=1}^{\ell} \binom{a_W(k)}{i_k},$$ where $\lam = \ell^{i_\ell}\dots 2^{i_2}1^{i_1}.$
\end{corollary}

The numbers $a_W(k)$ above can be computed recursively in two different ways, and the proofs of these recurrences will make clear why they are uniquely determined. The first (more efficient) way is as follows:
    
\begin{prop}\label{prop:a_recurrence}
    The exponents $a_W(k)$ satisfy the recurrence $$a_W(k) = [t^k]\log(I_{(G|_W,\omega)}(t)) + \sum_{d\mid k, \ d<k}\frac{(-1)^{k/d} d \cdot a_{W}(d)}k.$$
\end{prop}

Note that because of the division by $k$ and the non-integrality of the coefficients of $\log(I_{(G|_W,\omega)}(t)$, it is not immediately clear from Proposition \ref{prop:a_recurrence} whether the $a_W(k)$'s are integers. However, our second recursive formula does make it clear that they are integers:

\begin{prop}\label{prop:second_recurrence}
    The coefficients $a_W(k)$ also satisfy the alternative recurrence $$a_W(k) = [t^k]I_{(G|_W,\omega)}(t) - \sum_{\lam \vdash k, \ \lam\ne k} \ \prod_{j=1}^{k-1} \binom{a_W(k)}{i_j}.$$ Hence, $a_W(k)$ is an integer for all $W$ and $k.$
\end{prop}

An immediate consequence of Proposition \ref{prop:second_recurrence} and Corollary \ref{cor:p_coeff} is that the $\ol{p}$-coefficients are integers, resolving the question from \cite{kromatic2023}:

\begin{corollary}\label{cor:integrality}
    For every partition $\lam,$ the coefficient $[\ol{p}_\lam]\ol{X}_{(G,\omega)}$ is an integer.
\end{corollary}

Corollary \ref{cor:integrality} can also be proven in a different way by considering the transition formula between the $\ol{p}$-basis and another basis defined in \cite{kromatic2023}, the $\m$-basis. Hence, we will also give a short alternative proof of Corollary \ref{cor:integrality} in \S \ref{sec:p_expansion_proof} based on that transition formula.

\bigskip

Another question one can ask is about the signs of the coefficients. For $G$ an unweighted graph, it can be shown that the sign of $[p_\lam]X_G$ is $(-1)^{|\lam|-\ell(\lam)}$. In a more recent paper \cite{new_p_expansion}, we use Theorem \ref{thm:p_expansion} to give a combinatorial interpretation of the coefficients $[\ol{p}_\lam]\ol{X}_G,$ and we show that their signs follow the same pattern as the signs of $[p_\lam]X_G$, and that the signs of the exponents $a_W(k)$ alternate: 

\begin{theorem}\label{thm:signs}
    For $G$ unweighted, the sign of $[\ol{p}_\lam]\ol{X}_G$ is $(-1)^{|\lam|-\ell(\lam)}.$ Also, the sign of $a_W(k)$ is $(-1)^{k+1}.$
\end{theorem}

The first sentence of Theorem \ref{thm:signs} was stated as a conjecture in an earlier draft of this note, and that conjecture was suggested to the author by Oliver Pechenik. It does not apply in general to weighted graphs. For instance, in the $\ol{p}$-expansion of $K_{21}$ above, the coefficient for $\lam = 3$ is $-1$ while $|\lam|-\ell(\lam) = 3-1 = 2$ is even. The statement about the alternating signs of the $a_W(k)$'s also does not apply to unweighted graphs. In particular, we will show in \S \ref{sec:examples} that for $K_{21}$ and $W = V(K_{21}),$ $a_W(k)$ is positive for $k=2^j$, negative for $k=3\cdot 2^j,$ and 0 otherwise.

\bigskip

The remainder of this note is organized as follows. In \S \ref{sec:background}, we define some necessary terminology. In \S \ref{sec:examples}, we revisit the examples above to illustrate Theorem \ref{thm:p_expansion} and Propositions \ref{prop:a_recurrence} and \ref{prop:second_recurrence}. Then in \S \ref{sec:p_expansion_proof}, we prove Theorem \ref{thm:p_expansion}, Corollary \ref{cor:p_coeff}, Propositions \ref{prop:a_recurrence} and \ref{prop:second_recurrence}, and Corollary \ref{cor:integrality}.

\section{Background}\label{sec:background}

A \emph{\tb{\tcb{graph}}} $G$ is a set $V$ of \emph{\tb{\tcb{vertices}}} together with a set $E$ of \emph{\tb{\tcb{edges}}}, where each edge is an unordered pair of vertices. If $uv$ is an edge, we say that the vertices $u$ and $v$ are \emph{\tb{\tcb{adjacent}}}. A \emph{\tb{\tcb{weighted graph}}} $(G,\omega)$ is a graph $G$ together with a \emph{\tb{\tcb{weight function}}} $\omega:V\to \Z_{>0}$ that assigns a positive integer to each vertex. An \emph{\tb{\tcb{independent set}}} or \emph{\tb{\tcb{stable set}}} is a subset of $V$ in which no two vertices are adjacent. For $W\se V,$ the \emph{\tb{\tcb{induced subgraph}}} $G|_W$ with vertex set $v$ is the graph with vertex set $V(G|_W) := W$ and edge set $E(G|_W) := \{uv\in E(G):u,v\in W\}.$ That is, the edges of $G|_W$ are precisely the edges of $G$ with both endpoints in $W.$

A \emph{\tb{\tcb{symmetric function}}} $f(x_1,x_2,\dots)$ is a function that stays the same under any permutation of the variables, i.e. $f(x_1,x_2,\dots) = f(x_{\sigma(1)},x_{\sigma(2)},\dots)$ for any permutation $\sigma$ of $\Z_{>0}.$ A \emph{\tb{\tcb{proper coloring}}} of $V$ is a function $\alpha:V\to \Z_{>0}$ that assigns a positive integer valued color to each vertex such that $\alpha(u)\ne \alpha(v)$ whenever $uv \in E(G).$ The \emph{\tb{\tcb{chromatic symmetric function}}} is $$X_G := \sum_\alpha \prod_{v\in V} x_{\alpha(v)}$$ is the sum over all proper set colorings of the product of all the variables corresponding to colors used, counted with multiplicity.

A \emph{\tb{\tcb{partition}}} $\lam = \lam_1\lam_2\dots \lam_k$ is a nonincreasing sequence of positive integers $\lam_1 \ge \lam_2 \ge \dots \ge \lam_k,$ and $\lam_1,\dots,\lam_k$ are its \emph{\tb{\tcb{parts}}}. We write $|\lam| := \lam_1+\dots+\lam_k$ for the \emph{\tb{\tb{size}}} of $\lam$ and $\ell(\lam) = k$ for its \emph{\tb{\tcb{length}}} or number of parts. If $n=|\lam|,$ we say that $\lam$ is a \emph{\tb{\tcb{partition of $\boldsymbol{n}$}}} and write $\lam \vdash n.$ We can alternatively write $\lam = \ell^{i_\ell}(\ell-1)^{i_{\ell-1}}\dots 2^{i_2}1^{i_1}$ to denote the partition with $i_j$ parts of size $j$ for each $j = 1,2,\dots,\ell-1,\ell.$ For a partition $\lam=\lam_1\lam_2\dots \lam_k,$ the \emph{\tb{\tcb{power sum symmetric function}}} $p_\lam$ is $$p_\lam := \prod_{i=1}^{k} \left(\sum_{j\ge 1} x_j^{\lam_i}\right) = (x_1^{\lam_1} + x_2^{\lam_1} + \dots)(x_1^{\lam_2} + x_2^{\lam_2} + \dots)\dots(x_1^{\lam_k} + x_2^{\lam_k} + \dots).$$ 



A \emph{\tb{\tcb{proper set coloring}}} is a function $\alpha:V\to 2^{\Z_{>0}}\bs\{\varnothing\}$ that assigns a nonempty set of positive integer valued colors to each vertex such that adjacent vertices have nonoverlapping color sets, i.e. $\alpha(u)\cap \alpha(v) = \varnothing$ whenever $uv\in E(G).$ For a weighted graph $(G,\omega)$, the \emph{\tb{\tcb{Kromatic symmetric function}}} is $$\ol{X}_{(G,\omega)} := \sum_{\alpha} \prod_{v\in V}\prod_{i\in \alpha(v)}x_i^{\omega(v)}.$$ The authors of \cite{kromatic2023} define the $K$-analogue of the power sum symmetric functions as $$\ol{p}_\lam := \ol{X}_{\ol{K_\lam}},$$ where $\ol{K_\lam}$ is the edgeless graph whose vertex weights are the parts of $\lam.$

\section{Examples}\label{sec:examples}

We will now revisit the two example $\ol{p}$-expansions of $\ol{X}_{K_{21}}$ and $\ol{X}_{\imxs{figs/e}}$ from \S\ref{sec:intro} and show how to compute their coefficients using Theorem \ref{thm:p_expansion}, together with Proposition \ref{prop:a_recurrence} or \ref{prop:second_recurrence}.

\begin{example}
    Let $G = K_{21}$ consist of two connected vertices $v$ and $w,$ with $\omega(v) = 1$ and $\omega(w) = 2$. Then $$I_{(G,\omega)}(t) = 1 + t + t^2,$$ since $G$ has three independent sets: $\varnothing$ has weight 0, $\{v\}$ has weight 1, and $\{w\}$ has weight 2. By difference of cubes, $$I_{(G,\omega)}(t) = \frac{1-t^3}{1-t}.$$ Then by repeated difference of squares, $$\frac1{1-t} = (1+t)(1+t^2)(1+t^4)(1+t^8)\dots = \prod_{j\ge 0} (1+t^{2^j}),$$ and similarly, $$1-t^3 = \frac1{1+t^3}\cdot\frac1{1+t^6}\cdot\frac1{1+t^{12}}\cdot\frac1{1+t^{24}}\dots = \prod_{j\ge 0}\frac1{1+t^{3\cdot 2^j}},$$ which gives $$I_{(G,\omega)}(t) = \prod_{j\ge 0}(1+t^{2^j})\cdot\prod_{j\ge 0}\frac1{1+t^{3\cdot 2^j}}.$$ The other subsets of $V$ are $\{v\},$ $\{w\},$ and $\varnothing.$ In these cases, we get:
    \begin{align*}
        I_{(G|_{\{v\}},\omega)}(t) &= 1 + t, \\
        I_{(G|_{\{w\}},\omega)}(t) &= 1+t^2, \\
        I_{(G|_{\varnothing},\omega)}(t) &= 1.
    \end{align*}
    Putting this together, Theorem \ref{thm:p_expansion} gives $$\ol{X}_{(G,\omega)}(t) = \prod_{j\ge 0} (1+\ol{p}_{2^j})\cdot\prod_{j\ge0}\frac1{1+\ol{p}_{3\cdot 2^j}} - (1 + \ol{p}_1) - (1  + \ol{p}_2) + 1.$$ Note that in this case, the superscript $2^j$ refers to the partition with one part whose size is $2^j$ rather than the partition with $j$ parts of size 2. It follows that all $\ol{p}$-coefficients are $\pm 1$ or 0 in this case. In particular, $$[\ol{p}_\lam]\ol{X}_{(G,\omega)} = (-1)^{\#\tn{ of parts of }\lam\tn{ that are multiples of 3}}$$ provided that $\lam$ satisfies all of the following properties:
    \begin{itemize}
        \item $|\lam| \ge 3.$
        \item All parts are of the form $2^j$ or $3\cdot 2^j$ for some $j\ge 0.$
        \item No parts of the form $2^j$ are repeated.
    \end{itemize}
    All other coefficients are 0. To illustrate this, we can expand out the first few terms to get
    \begin{align*}
        \ol{X}_{(G,\omega)}(t) &= (1 + \ol{p}_1)(1 + \ol{p}_2)(1 + \ol{p}_4)\dots\frac1{1+\ol{p}_3}\cdot\frac1{1+\ol{p}_6}\cdot\frac1{1+\ol{p}_{12}}\cdot \dots - (1 + \ol{p}_1 + \ol{p}_2) \\
        &= (1 + \ol{p}_1)(1 + \ol{p}_2)(1 + \ol{p}_4)\dots(1 - \ol{p}_3 + \ol{p}_{33} - \ol{p}_{333} + \dots)(1 - \ol{p}_6 + \ol{p}_{66} - + \dots)\dots - (1 + \ol{p}_1 + \ol{p}_2) \\
        &= (- \ol{p}_3 + \ol{p}_{21} ) + (\ol{p}_4 - \ol{p}_{31}) + (\ol{p}_{41} - \ol{p}_{32}) + (- \ol{p}_6 + \ol{p}_{42} + \ol{p}_{33} - \ol{p}_{321}) \\
        & \ \ \ + (-\ol{p}_{61} + \ol{p}_{421} + \ol{p}_{331}) + (\ol{p}_8 - \ol{p}_{62} + \ol{p}_{332}) + (\ol{p}_{81} + \ol{p}_{63} - \ol{p}_{621} - \ol{p}_{333} + \ol{p}_{3321}) + \dots
    \end{align*}
    We can see that these terms match the ones from \S \ref{sec:intro}.
\end{example}

For the seemingly simpler case of a single \emph{unweighted} edge $\ims{figs/e}$, there is not such a simple formula for the coefficients:

\begin{example}
    Let $G = \ims{figs/e} = \{v,w\},$ where $v$ and $w$ both have weight 1. Then we get $$I_G(t) = 1+2t,$$ since $\{v\}$ and $\{w\}$ are stable sets of size 1 and $\varnothing$ is a stable set of size 2. In this case the factorization begins $$1+2t = (1+t)^2\cdot\frac1{1+t^2}\cdot(1+t^3)^2\cdot\frac1{(1+t^4)^4}\cdot(1+t^5)^6 \dots.$$ Thus, for $W = V,$ the sequence of exponents $a_W(k)$ begins $$a_W(1) = 2, \ \ a_W(2) = -1, \ \ a_W(3) = 2, \ \ a_W(4) = -4, \ \ a_W(5) = 6,\dots$$ These numbers match sequence A038067 in the Online Encyclopedia of Integer Sequences \cite{oeis}. They can be computed recursively using either Proposition \ref{prop:a_recurrence} or Proposition \ref{prop:second_recurrence}:
    \begin{itemize}
        \item To use Proposition \ref{prop:a_recurrence}, we first note that 
        \begin{align*}
            \log(I_G(t)) = \log(1+2t) &= 2t - \frac{(2t)^2}{2} + \frac{(2t)^3}{3} - \frac{(2t)^4}{4} + \frac{(2t)^5}{5} - \cdots \\
            &= 2t - 2t^2 + \frac{8t^3}3 - 4t^4 + \frac{32t^5}{5}-\cdots.
        \end{align*} Thus we get
        \begin{align*}
            a_W(1) &= [t]\log(I_G(t))  = 2, \\
            a_W(2) &= [t^2]\log(I_G(t)) + \frac{(-1)^{2/1}\cdot 1\cdot a_W(1)}{2} = -2 + 1 = -1, \\
            a_W(3) &= [t^3]\log(I_G(t)) + \frac{(-1)^{3/1}\cdot 1\cdot a_W(1)}{3} = \frac83-\frac23 = 2, \\
            a_W(4) &= [t^4]\log(I_G(t)) + \frac{(-1)^{4/2}\cdot2\cdot a_W(2)}{4} + \frac{(-1)^{4/1}\cdot1\cdot a_W(1)}{4} = -4 -\frac12+\frac12 = -4, \\ 
            a_W(5) &= [t^5]\log(I_G(t)) + \frac{(-1)^{5/1}\cdot1\cdot a_W(1)}{5} = \frac{32}5 - \frac25 = 6.
        \end{align*}
        \item Alternatively, to use Proposition \ref{prop:second_recurrence}, we list the partitions $k$ for each $k=1,2,3,4,5$:
        \begin{itemize}
            \item $\boldsymbol{k=1:}$ 1,
            \item $\boldsymbol{k=2:}$ 2, $1^2$,
            \item $\boldsymbol{k=3:}$ 3, 21, $1^3$,
            \item $\boldsymbol{k=4:}$ 4, 31, $2^2$, $21^2,$ $1^4$,
            \item $\boldsymbol{k=5:}$ 5, 41, 32, $31^2$, $2^21,$ $21^3,$ $1^5.$
        \end{itemize}
        Then we apply Proposition \ref{prop:second_recurrence} to compute
        \begin{align*}
            a_W(1) &= [t]I_G(t) = 2, \\
            a_W(2) &= [t^2]I_G(t) - \binom{a_W(1)}{2} = 0 - \binom{2}{2} = 0 - 1 = -1, \\
            a_W(3) &= [t^3]I_G(t) - \binom{a_W(2)}{1}\binom{a_W(1)}{1} - \binom{a_W(1)}{3} \\
            &= 0 - \binom{-1}{1}\binom{2}{1} - \binom{2}{3} \\
            &= 0 - (-1)(2) - 0 = 2, \\
            a_W(4) &= [t^4]I_G(t) - \binom{a_W(3)}{1}\binom{a_W(1)}{1} - \binom{a_W(2)}{2} - \binom{a_W(2)}{1}\binom{a_W(1)}{2} - \binom{a_W(1)}{4} \\
            &= 0 - \binom{2}{1}\binom{2}{1} - \binom{-1}{2} - \binom{-1}{1}\binom{2}{2} - \binom{2}{4} \\
            &= 0 - (2)(2) - 1 - (-1)(1) - 0 = -4, \\
            a_W(5) &= [t^5]I_G(t) - \binom{a_W(4)}{1}\binom{a_W(1)}{1} - \binom{a_W(3)}{1}\binom{a_W(2)}{1} - \binom{a_W(3)}{1}\binom{a_W(1)}{2} \\ & \hspace{1.5cm} \ \ -\binom{a_W(2)}{2}\binom{a_W(1)}{1}
            - \binom{a_W(2)}{1}\binom{a_W(1)}{3} - \binom{a_W(1)}{5} \\
            &= 0 - \binom{-4}{1}\binom{2}{1} - \binom{2}{1}\binom{-1}{1} - \binom{2}{1}\binom{2}{2} - \binom{-1}{2}\binom{2}{1} - \binom{-1}{1}\binom{2}{3} - \binom{2}{5} \\
            &= 0 - (-4)(2) - (2)(-1) - (2)(1) - (1)(2) - (-1)(0) - 0 = 8 + 2 - 2 - 2 = 6.
        \end{align*}
    \end{itemize}
    Both methods give us the same exponents $a_W(1) = 2, \ a_W(2) = -1, \ a_W(3) = 2, \ a_W(4) = -4, \ a_W(5) = 6$ that we listed previously.

    \bigskip
    
    To finish computing the $\ol{p}$-expansion of $\ol{X}_G$, we note that the other stable sets are $\{v\},$ $\{w\}$, and $\varnothing,$ which have independence polynomials $$I_{G|_{\{v\}}}(t) = I_{G|_{\{w\}}}(t) = 1+t, \ \ \ \ I_{G|_{\varnothing}}(t) = 1.$$ Putting this together, Theorem \ref{thm:p_expansion} gives \begin{align*}\ol{X}_G(t) &= (1+\ol{p}_1)^2\cdot\frac1{1+\ol{p}_2}\cdot(1+\ol{p}_3)^2\cdot\frac1{(1+\ol{p}_4)^4}\cdot(1+\ol{p}_5)^6\dots - (1 + \ol{p}_1) - (1 + \ol{p}_1) + 1 \\
    &= (1 + 2\ol{p}_1 + \ol{p}_{11})(1 - \ol{p}_2 + \ol{p}_{22} - \dots)(1 + 2\ol{p}_3 + \ol{p}_{33})(1 - 4\ol{p}_4 + 10\ol{p}_{44} - 20\ol{p}_{444} + \dots)\dots - (1 + 2\ol{p}_1) \\
    &= (- \ol{p}_2 + \ol{p}_{11}) + (2\ol{p}_3 - 2\ol{p}_{21}) + (-4\ol{p}_4 + 4 \ol{p}_{31} + \ol{p}_{22} - \ol{p}_{211}) \\
    & \ \ \ + (6\ol{p}_5 - 8 \ol{p}_{41} - 3\ol{p}_{32} + 2 \ol{p}_{311} + 2\ol{p}_{221}) \\
    & \ \ \ + (-9\ol{p}_6 + 12\ol{p}_{51} + 4\ol{p}_{42} - 4\ol{p}_{411} + \ol{p}_{33} - 4\ol{p}_{321} - \ol{p}_{222} + \ol{p}_{2211})\dots,
    \end{align*}
    matching the expansion from \S \ref{sec:intro}.
\end{example}

\section{Proofs}\label{sec:p_expansion_proof}

In this section we will prove Theorem \ref{thm:p_expansion}, Corollary \ref{cor:p_coeff}, Propositions \ref{prop:a_recurrence} and \ref{prop:second_recurrence}, and Corollary \ref{cor:integrality}.

\begin{proof}[Proof of Theorem \ref{thm:p_expansion}]
Our key lemma is an expansion formula for Stanley's $Y_G$ function (introduced in \cite{stanley1998graph}), where $Y_{(G,\omega)}$ is defined to be the same as $\ol{X}_{(G,\omega)}$ except that it includes monomials corresponding to colorings where some vertices are not assigned any colors: $$Y_{(G,\omega)} := \sum_\kappa \prod_{v\in V}\left(\prod_{i\in \kappa(v)} x_i\right)^{\omega(v)},$$ where the sum ranges over all functions $\kappa:V\to 2^n$ such that $\kappa(v)\cap \kappa(w) =\varnothing$ whenever $vw\in E(G).$

\begin{lemma}\label{lem:Y_G}
    $Y_{(G,\omega)}$ factors as $$Y_{(G,\omega)} = \prod_{k\ge 1} (1 + \ol{p}_k)^{a_{V}(k)}.$$
\end{lemma}

\begin{proof}
    Since $Y_{(G,\omega)}$ allows vertices to not receive any colors, we can think of assigning each color to a subset of the vertices independently of all other colors. Specifically, for each color, we can independently choose a subset of $V$ which will receive that color, and the only restriction is that the subset be an independent set. Thus, the generating series for the ways to choose which vertices receive color $i$ is given by the independence polynomial $I_{(G,\omega)}(x_i)$. Since different colors can be assigned independently, we can compute $Y_{(G,\omega)}$ by simply multiplying over all colors, so $$Y_{(G,\omega)} = \prod_{i\ge 1} I_{(G,\omega)}(x_i).$$ Then by the definition of $a_{V}(k)$, we have $$I_{(G,\omega)}(x_i) = \prod_{k\ge 1} (1+x_i^k)^{a_{V}(k)}.$$ Now note that $\ol{p}_k$ represents the sum of all monomials corresponding to ways to assign a nonempty set of colors to a single weight $k$ vertex, so $1+\ol{p}_k$ represents the sum of monomials corresponding to all ways to assign a set of colors to the weight $k$ vertex, including the empty set. To choose such a coloring, we can independently choose whether or not to use each color. If color $i$ is used, then we get an $x_i^k$ in the monomial, and otherwise the power of $x_i$ is 0. Since we can independently choose an exponent of 0 or $k$ for each variable, $$1+\ol{p}_k = \prod_{i\ge 1}(1+x_i^k).$$ It follows that $$Y_{(G,\omega)} = \prod_{i\ge 1} I_{(G,\omega)}(x_i) = \prod_{i,k\ge 1}(1+x_i^k)^{a_{V}(k)} = \prod_{k\ge 1}(1+\ol{p}_k)^{a_{V}(k)},$$ as claimed.
\end{proof}

Theorem \ref{thm:p_expansion} then follows from the principle of inclusion-exclusion. The difference between $\ol{X}_{(G,\omega)}$ and $Y_{(G,\omega)}$ is that $\ol{X}_{(G,\omega)}$ requires all vertices to receive a nonempty set of colors. Thus, by inclusion-exclusion, we can compute $\ol{X}_{(G,\omega)}$ by adding the monomials all possible colorings without that restriction, then subtracting the monomials where a particular vertex receives no colors, then adding back the ones where a particular pair of vertices receive no colors, and so on. It follows that $$\ol{X}_{(G,\omega)} = \sum_{W\se V} Y_{(G,\omega)}.$$ Plugging in the formula from Lemma \ref{lem:Y_G} proves Theorem \ref{thm:p_expansion}.
\end{proof}

The formula in Corollary \ref{cor:p_coeff} is now fairly immediate:

\begin{proof}[Proof of Corollary \ref{cor:p_coeff}]
    Let $\lam = \ell^{i_\ell}\dots 2^{i_2}1^{i_1}.$ First, we note that $$\ol{p}_\lam = \ol{p}_{\lam_1}\dots \ol{p}_{\lam_{\ell(\lam)}} = \ol{p}_\ell^{i_\ell}\dots \ol{p}_2^{i_2}\ol{p}_1^{i_1}.$$ To see this, recall that $\ol{p}_\lam$ is the sum of all monomials corresponding to colorings of the vertices of the edgeless graph $\ol{K_\lam}$ with vertex weights $\lam_1,\lam_2,\dots,\lam_{\ell(\lam)}$, or equivalently, with $i_k$ vertices of weight $k$ for each $k=1,2,\dots,\ell.$ Since the graph has no edges, the set of colors assigned to a vertex has no bearing on how we can color the other vertices, so we can choose the coloring of each vertex independently. Thus, $\ol{p}_\lam$ is the product over all vertices of the sum of all possible monomials for ways to color that vertex, which is just $\ol{p}_k$ for a vertex of weight $k.$
    
    Thus, to get a $\ol{p}_\lam$ term from $\prod_{k\ge 1}(1+\ol{p}_k)^{a_W(k)},$ we need to choose a $\ol{p}_k^{i_k}$ term from $(1+\ol{p}_k)^{a_W(k)}$ for each $k=1,2,\dots,\ell.$ By the binomial theorem, $$(1+\ol{p}_k)^{a_W(k)} = \sum_{i\ge 0} \binom{a_W(k)}{i}\ol{p}_k^i,$$ so the coefficient of $\ol{p}_\lam$ in the product $\prod_{k\ge 1}(1+\ol{p}_k)^{a_W(k)}$ is just the product of these coefficients: $$[\ol{p}_\lam]\prod_{k\ge 1}(1+\ol{p}_k)^{a_W(k)} = \prod_{k=1}^\ell \binom{a_W(k)}{i_k}\ol{p}_k^i.$$ Summing over all $W\se V$ completes the proof.
\end{proof}

Next, we prove the recursive formula in Proposition \ref{prop:a_recurrence}:

\begin{proof}[Proof of Proposition \ref{prop:a_recurrence}]
    Taking logarithms of both sides of the equation $$(1+t)^{a_W(1)}(1+t^2)^{a_W(2)}(1+t^3)^{a_W(3)}\dots = I_{(G|_W,\omega)}(t)$$ gives $$\sum_{j\ge 1}a_W(j)\log(1+t^j) = \log(I_{G|_W,\omega}(t)).$$ Expanding the left side using the power series $\log(1+x) = \displaystyle{\sum_{i\ge 1} \frac{(-1)^{i+1}x^i}{i}},$ we get $$\sum_{j\ge 1}a_W(j)\sum_{i\ge 1}\frac{(-1)^{i+1}t^{ij}}{i} = \log(I_{G|_W,\omega}(t)).$$ Since $j=k$ is the largest value of $j$ on the left side that contributes a $t^k$ term, we can compute $a_W(k)$ recursively from the values $a_W(j)$ for $j < k$ by making the coefficient of $t^k$ on the left side match the coefficient $[t^k]\log(I_{G|_W,\omega}(t))$ on the right side. The values $j < k$ that contribute a $t^k$ term on the left are precisely the values where $j$ is some divisor $d$ of $k$, in which case $i = k/d.$ The contribution from the $j=d$ term is then $$a_W(d)\frac{(-1)^{k/d+1}}{k/d} = \frac{(-1)^{k/d}d\cdot a_W(d)}{k}.$$ Summing over all values $d\mid k$ and then subtracting the resulting sum from the right side coefficient $[t^k]\log(I_{G|_W,\omega}(t))$ gives the formula in Proposition \ref{prop:a_recurrence}.
\end{proof}

Now we prove the alternative recurrence from Proposition \ref{prop:second_recurrence}:

\begin{proof}[Proof of Proposition \ref{prop:second_recurrence}]
    Assume we know $a_W(1),\dots,a_W(k-1)$ and want to find $a_W(k).$ In the product $$\prod_{j\ge 1}(1+t^j)^{a_W(j)} = I_{(G|_W,\omega)}(t),$$ the coefficient $[t^k]$ of $t^k$ needs to be the same on both sides. On the right side, the coefficient is $[t^k]I_{(G|_W,\omega)}(t).$ On the left side, since $$(1+t^k)^{a_W(k)} + 1 + a_W(k)\cdot t^k + \dots,$$ the $(1+t^k)^{a_W(k)}$ factor contributes $a_W(k)$ to $[t^k].$ No factor $(1+t^j)^{a_W(j)}$ with $j > k$ can contribute to $[t^k].$ Each remaining contribution come from some partition $\lam = \ell^{i_\ell}\dots 2^{i_2}1^{i_1} \vdash k$ and corresponds to taking a $(t^j)^{i_j}$ from the $(1+t^j)^{a_W(j)}$ factor for each $j=1,2,\dots,\ell.$ By the binomial theorem, the coefficient of $(t^j)^{i_j}$ in $(1+t^j)^{a_W(j)}$ is $\binom{a_W(j)}{i_j},$ so the overall coefficient of $(t^\ell)^{i_\ell}\dots (t^2)^{i_2}(t^1)^{i_1}$ is $$\binom{a_W(\ell)}{i_\ell}\dots \binom{a_W(2)}{i_2}\binom{a_W(1)}{i_1}.$$ We then subtract all those coefficients from $[t^k]I_{(G|_W,\omega)}(t)$ to give the formula in Proposition \ref{prop:second_recurrence}.
\end{proof}

As noted in \S \ref{sec:intro}, Corollary \ref{cor:integrality} about the integrality of the $\ol{p}$-coefficients is an immediate consequence of Corollary \ref{cor:p_coeff} and Proposition \ref{prop:second_recurrence}, but we will now show that an alternative proof can also be given by considering the transition formula between the $\ol{p}$-basis and another basis defined in \cite{kromatic2023}: the \emph{\tb{\tcb{monomial symmetric function}}} $m_\lam$ is $$m_\lam := \sum_{n_1,\dots,n_k} x_{n_1}^{\lam_1} \dots x_{n_k}^{\lam_k},$$ where the sum is taken over all ordered tuples of $k$ distinct indices $n_1,\dots,n_k \in \Z_{>0}$. The \emph{\tb{\tcb{augmented monomial symmetric function}}} $\til{m}_\lam$ is $$\til{m}_\lam := i_1!\dots i_\ell!\cdot m_\lam,$$ where $i_j$ is the number of parts of $\lam$ of size $j.$ The \emph{\tb{\tcb{$K$-theoretic augmented monomial symmetric functions}}} $\m_\lam$ is $$\m_\lam := \ol{X}_{K_\lam},$$ where $K_\lam$ is a complete graph whose vertex weights are the parts of $\lam.$ A \emph{\tb{\tcb{stable set cover}}} $C = \{S_1,\dots,S_k\}$ is a set of $k$ independent sets in $G$ such that every vertex is in at least one of $S_1,\dots,S_k.$ We write $\textsf{SSC}(G)$ for the set of all stable set covers of $G$, and $\lam(C)$ for the partition whose parts are the weights $\omega(S_1),\dots,\omega(S_k),$ where the \emph{\tb{\tcb{weight}}} of a subset is the sum of the weights of its vertices, $\omega(S) := \sum_{v\in S} \omega(v).$ The authors of \cite{kromatic2023} prove the following $\m$-expansion formula for $\ol{X}_G$:

\begin{theorem}[{Crew, Pechenik, and Spirkl \cite[Proposition 3.4]{kromatic2023}}]
    The $\m$-expansion for $\ol{X}_G$ is $$\ol{X}_{(G,\omega)} = \sum_{C \in \textsf{\emph{SSC}}(G)} \m_{\lam(C)}.$$
\end{theorem}

We can now use this to give our second proof of Corollary \ref{cor:integrality}:

\begin{proof}[Alternative proof of Corollary \ref{cor:integrality}]
    Consider the $\m$-expansion of $\ol{p}_\lam = \ol{X}_{\ol{K_\lam}}.$ Since $\ol{K_\lam}$ has no edges, all subsets of its vertices are independent sets. One possible stable set cover is the one where each vertex is in a stable set by itself, in which case $\lam(C) = \lam,$ so the expansion of $\ol{p}_\lam$ has an $\m_\lam$ term. For every other term $\m_\mu$, we must have either $|\mu| > |\lam|$ (if some vertices are covered by multiple stable sets in $C$) or else $\mu$ is formed by taking $\lam$ and combining some of its parts (if each vertex is only in one stable set but some stable sets contain multiple vertices), which is equivalent to saying that $\lam$ is a \emph{\tb{\tcb{refinement}}} of $\mu$. Thus, we have $[\m_\lam]\ol{p}_\lam=1$ and $[\m_\mu]\ol{p}_\mu=0$ whenever $|\mu|<|\lam|$ or $|\mu|=|\lam|$ but $\lam$ is not a refinement of $\mu.$ Now, order the set of all partitions in such a way that partitions of smaller numbers come first and $\lam$ comes before $\mu$ whenever $\lam$ is a refinement of $\mu$, which can be done by ordering the partitions of each size lexicographically. Then the transition matrix from the $\ol{p}$-basis to the $\m$-basis is upper triangular with 1's along the diagonal and all integer entries. It follows that the inverse matrix which transitions from the $\m$-basis to the $\ol{p}$-basis is also upper triangular with 1's on the diagonal and all integer entries. Thus, each $\m_\mu$ is a linear combination of $\ol{p}_\lam$ terms with integer coefficients. Since $\ol{X}_{(G,\omega)}$ is an integral linear combination of $\m_\mu$ terms for every weighted graph $(G,\omega)$, it follows that $\ol{X}_{(G,\omega)}$ is an integral linear combination of $\ol{p}_\lam$ terms, as claimed.
\end{proof}

\section*{Acknowledgements}

The author thanks Oliver Pechenik for suggesting the problem, and she thanks Oliver Pechenik, Karen Yeats, Sophie Spirkl, and the anonymous reviewers for providing helpful comments. She was partially supported by the Natural Sciences and Engineering Research Council of Canada (NSERC) grant RGPIN-2022-03093.

\printbibliography

\end{document}